\documentclass[oneside,english]{amsart}
\usepackage[latin9]{inputenc}
\usepackage{amsthm}
\usepackage{amssymb}

\makeatletter

\newif\ifmargin
\margintrue

\let\SF@@footnote\footnote
\def\footnote{\ifx\protect\@typeset@protect
    \expandafter\SF@@footnote
  \else
    \expandafter\SF@gobble@opt
  \fi
}
\expandafter\def\csname SF@gobble@opt \endcsname{\@ifnextchar[
  \SF@gobble@twobracket
  \@gobble
}
\edef\SF@gobble@opt{\noexpand\protect
  \expandafter\noexpand\csname SF@gobble@opt \endcsname}
\def\SF@gobble@twobracket[#1]#2{}

\numberwithin{equation}{section} 
\numberwithin{figure}{section} 
\theoremstyle{plain}
\theoremstyle{plain}
\newtheorem{thm}{Theorem}
  \theoremstyle{plain}
  \newtheorem{lem}[thm]{Lemma}
  \theoremstyle{plain}
  \newtheorem{prop}[thm]{Proposition}
  \theoremstyle{plain}
  \newtheorem{cor}[thm]{Corollary}
  \theoremstyle{definition}
  \newtheorem{defn}[thm]{Definition}

\makeatother

\usepackage{babel}

\begin{document}

\title{On the reconstruction of graph invariants}

\author{T. Kotek}

\thanks{Partially supported by ISF-Grant \textquotedblright{}Model Theoretic
Interpretations of Counting Functions\textquotedblright{}, 2007-2010. }

\curraddr{Department of Computer Science\\
Technion --- Israel Institute of Technology\\
Haifa, Israel}

\email{tkotek@cs.technion.ac.il}

\maketitle
\global\long\def\mc#1{\mathcal{#1}}

\global\long\def\angl#1{\left\langle #1\right\rangle }

\global\long\def\chp{\chi_{p-h}}

\global\long\def\chh{\chi_{harm}}


\begin{abstract}
The reconstruction conjecture has remained open for simple undirected
graphs since it was suggested in 1941 by Kelly and Ulam. In an attempt
to prove the conjecture, many graph invariants have been shown to
be reconstructible from the vertex-deleted deck, and in particular,
some prominent graph polynomials. Among these are the Tutte polynomial,
the chromatic polynomial and the characteristic polynomial. We show that the interlace polynomial, the $U$-polynomial,
the universal edge elimination polynomial $\xi$ and the colored versions of the latter two are
reconstructible.
\end{abstract}

\section{Introduction}
\subsection{The Reconstruction Conjecture}
Let $G=\left(V(G),E(G)\right)$ be a simple undirected graph. For
a vertex $v\in V(G)$, we define the vertex-deleted subgraph $G_{v}$
as the subgraph obtained from $G$ by deleting $v$ and all its incident
edges. We define the \emph{deck} of $G$, $\mathcal{D}(G)$, as the collection
of subgraphs $G_v$ for every $v\in V(G)$. The deck $\mathcal{D}(G)$
is a multiset which consists of isomorphism types of subgraphs. We
say a graph $H$ is a \emph{reconstruction} of $G$ if $V(G)=V(H)$
and for every $v\in V(G)$, $G_{v}\cong H_{v}$. A graph $G$ is \emph{reconstructible
}if every reconstruction of it is isomorphic to it. 

The~reconstruction~conjecture asserts that every simple undirected
graph with at least three vertices is reconstructible. We say a graph invariant
$\mu$ is reconstructible if for every $H$ and $G$ which are reconstructions
of each other, $\mu(G)=\mu(H)$. If $\mu$ is the characteristic function of a class of graphs $\mathcal{C}$, we say that $\mathcal{C}$ is recognizable. 

For graphs $H$ and $G$ we define $s(H,G)$ as the number of subgraphs of $G$ which
are isomorphic to $H$. 
\begin{lem}
[P. J. Kelly, 1957] For any $H$ and $G$ such that $|V(H)|<|V(G)|$,
$s(H,G)$ is reconstructible from the deck of $G$. 
\end{lem}
Another useful lemma by Tutte is the following (see \cite{pr:Tutte79}).
\begin{lem}
\label{lemmaA}The number of disconnected spanning subgraphs of $G$ having a specified
number of components in each isomorphism class and the number of connected spanning subgraphs with a given number of
edges are reconstructible. 
\end{lem}
An edge-colored graph $G$ is $G=\left(V(G),E(G),f\right)$ where
$f:E(G)\to[\Lambda]$ is a coloring of the edges of $G$. Similarly to
the deck of a simple undirected graph, we define the deck $\mathcal{D}(G)$
of a colored graph $G$ to be the multiset of colored graphs $G_{v}$
for all $v\in V(G)$ which consist of $G$ with $v$ and all incident
edges deleted from it, and $f$ restricted to $V(G)-\left\{ v\right\} $.
Reconstruction results on graphs
can often be adapted to colored graphs, and in particular, Kelly's
lemma works for colored graphs, see \cite{ar:MW78,ar:Weinstein75}.

\subsection{Graph polynomials}
Graph polynomials are functions $F$ from the class of simple undirected graphs to a polynomial ring $\mathcal{R}[\bar{x}]$ that are invariant under graph isomorphism. They  are natural to the study of reconstruction.
The characteristic polynomial $P(G,x)$ is the characteristic polynomial
of the adjacency matrix $A_{G}$ of the graph $G$, 
$P(G;x)=det(x\cdot\mathbf{1}-A_{G})=\sum_{S\subseteq V(G)} (-x)^{|V(G)|-|S|} det( A_{G_S})$.
Tutte showed a generalization
of the characteristic polynomial is reconstructible in \cite{pr:Tutte79}. 
The Tutte polynomial is 
$T(G;x,y)=\sum_{M\subseteq E(G)}\left(x-1\right)^{r(E)-r(M)}\left(y-1\right)^{n(M)}$,
 where $k(M)$ is the number of connected components of the spanning
subgraph $\left(V(G),M\right)$, $r(M)=|V(G)|-k(M)$ is its rank and
$n\left(M\right)=|M|-|V(G)|+k\left(M\right)$ is its nullity. The Tutte polynomial was
shown to be reconstructible by Tutte. The
chromatic polynomial $\chi(G,\lambda)$, which counts the number of
proper colorings of $G$ with at most $\lambda$ colors, is an evaluation
of the Tutte polynomial, and so is reconstructible as well.

\section{\label{secLab}Main Results}
\subsection{Reconstructibility of graph polynomials}
For a graph $G$ and vertex set $S$, we denote by $G[S]$ the subgraph
induced by $S$. The polynomial
$q(G;x,y)=\sum_{S\subseteq V(G)}\left(x-1\right){}^{rk\left(G[S]\right)}\left(y-1\right)^{n\left(G[S]\right)}$ 
where $rk(G[S])$ is the rank of $G[S]$'s adjacency matrix, and $n(G[S])+rk(G[S])=|S|$ is the 2-variable interlace polynomial (\cite{ar:ABS04}). 
\begin{prop}
The interlace polynomial is reconstructible.
\end{prop}
\begin{proof}{Proof (Sketch)}
The determinant of $A_G$ is reconstructible (see \cite{pr:Tutte79}). $rk(V(G))=|V(G)|$ iff $det(A_G)\not=0$, and if $rk(V(G))\not=|V(G)|$, $rk(G)=max_{v\in V(G)} rk(A_{G_v})$. So we can reconstruct the term for $S=V(G)$. The rest are reconstructible by Kelly's lemma. 
\end{proof}

The $U$ polynomial 
$U(G;\bar{x},y)=\sum_{A\subseteq E(G)}y^{|A|-r(A)}\prod_{D:\Phi_{cc}(D)}x_{|V(D)|},$  where the product iterates over the connected components of the graph $(V,A)$, and $r(A)=|V(G)|-k(A)$, was introduced in \cite{ar:nb99}.  The $U(G;\bar{x},y)$ polynomial
generalizes both the Tutte polynomial and a graph polynomial related 
to Vassiliev invariants of knots. The graph polynomial on colored graphs
$
U_{lab}(G;\bar{x},\bar{y})=\sum_{A\subseteq E}\prod_{e\in A}y_{c(e)}\cdot^{}\prod_{D:\Phi_{cc}(D)}x_{|V(D)|}$
is a generalization of $U$ to colored graphs. We show that $U_{lab}$ is reconstructible (see also \cite{ar:Kocay81}):
\begin{thm}\label{UReconColor}
The $U_{lab}$-polynomial is reconstructible. 
\end{thm}

\begin{proof}{Proof (Sketch)}
 Let $G$ be a colored graph with color-set $\left[\Lambda\right]$.
For $c\in\left[\Lambda\right]$ we denote by $f^{-1}(c)$ the set
of edges colored $c$, and note $|f^{-1}(c)|$ is reconstructible by Kelly's lemma. 
For a tuple $\left(a_{1},\ldots a{}_{\Lambda}\right)\in\mathbb{N}^{\Lambda}$,
we say a subgraph $A\subseteq E(G)$ is $\left(a_{1},\ldots a{}_{\Lambda}\right)$-colored
if for every color $c\in\left[\Lambda\right]$, there are $a_{c}$
many edges colored $c$. For every $\left(a_{1},\ldots a{}_{\Lambda}\right)\in\mathbb{N}^{\Lambda}$,
the number of $\left(a_{1},\ldots a{}_{\Lambda}\right)$-colored subgraphs
is $\prod_{c\in\left[\Lambda\right]}{|f^{-1}(c)| \choose a_{c}}$. 

We reconstruct the terms of $U_{lab}$ that correspond to subgraphs which are
not spanning subgraphs by Kelly's lemma. 
For disconnected spanning subgraphs, we go over all tuples
$\left(F_{1},\ldots,F_{t}\right)$, $t>1$, of isomorphism types of
colored graphs such that the number of vertices in all the $F_{i}$'s
is $\sum_{i=1}^{t}|V(F_{i})|=|V(G)|$. For each such tuple we find
the number of ways of choosing subgraphs $G_{1},\ldots,G_{t}$ of $G$ such that $G_{i}\cong F_{i}$ for every $i$, and their
union $\bigcup_{i=1}^{t}G_{i}$ is $G.$ We can do this by subtracting the number of ways to choose subgraphs as above such that their union does not span $V(G)$ from $s(F_{1},G)\cdots s(F_{t},G)$, the number of ways to choose the $G_i$ without restriction. 

 For any $\left(a_{1},\ldots a,_{\Lambda}\right)\in\mathbb{N}^{\Lambda}$,
the number of $\left(a_{1},\ldots a,_{\Lambda}\right)$-colored spanning
connected subgraphs is $\prod_{c\in\left[\Lambda\right]}{|f^{-1}(c)| \choose  a_{c}}$
minus the number of $\left(a_{1},\ldots a,_{\Lambda}\right)$-colored
subgraphs which are not spanning connected subgraphs. For them, the
corresponding term is $x_{|V(G)|}\prod_{c\in\left[\Lambda\right]}y_{c}^{a_{c}}.$ 

\end{proof}

\begin{cor} Lemma \ref{lemmaA} holds for colored graphs. 
\end{cor}

The universal edge elimination polynomial $\xi(G;x,y,z)$ introduced in \cite{pr:agm08}
is defined with respect to three basic operations on edges of the
graph: edge contraction $G_{/e}$, edge deletion $G_{-e}$ and edge elimination $G_{\dagger e}$. It is the most general graph polynomial which satisfies a linear recurrence
relation with respect to the above operations which is multiplicative
with respect to disjoint union and has initial conditions $\xi(E_{1})=x$ and $\xi(\emptyset)=1$. The
$\xi$ polynomial generalizes a number of graph polynomials and graph
parameters which satisfy linear recurrences with respect to these
operations, including the bivariate matching polynomial, the Tutte
polynomial, and the bivariate chromatic polynomial (the latter defined
in \cite{ar:DPT03}). The authors of \cite{pr:agm08} also defined a version $\xi_{lab}$ for colored graphs, given by 
$
 \xi_{lab}(G;x,z,\bar{t})=\sum_{A\sqcup B\subseteq E(G)} x^{k(A\sqcup B)} \left( \prod_{e\in A\sqcup B} t_{c(e)} \right) z^{k_{cov}(B)},
$
where the summation is over subsets $A$ and $B$ of $E(G)$ which
cover disjoint subsets of vertices $V(A)$ and $V(B)$, $k_{cov}(B)$
denotes the number of connected components in the graph $(V(B),B)$,
and $k(A\sqcup B)$ denotes the number of connected components in
$(V(G),A\sqcup B)$, the graph which consists of the vertices of $G$
and the edges of $A$ and $B$.

\begin{thm}
$\xi_{lab}$ and its evaluations, including the colored Sokal polynomial, Zaslavsky's
normal function of the colored matroid, the chain polynomial and the
multivariate matching polynomial, are all reconstructible.
\end{thm}
\begin{proof}{Proof (Sketch)}
A monomial of $U_{lab}$ for some $M\subseteq E(G)$
corresponds in $\xi_{lab}$ to monomials for which $A\sqcup B=M$.
Since $A$ and $B$ are disjoint in vertices, every connected
component must be either entirely in $A$ or entirely in $B$. We can reconstruct the number of connected components spanned by $M$ as the number of $x$ variables in the monomial and the number of vertices spanned by $M$ as the sum of the indices of the $x$ variables. Now we need only choose for each connected component whether it is in $A$ or in $B$. 
\end{proof}

Restricting to uncolored graphs we have:

\begin{thm}
The $U$ and $\xi$ polynomials and all of their evaluations, including the bivariate chromatic polynomial, the chromatic
symmetric function $X_{G}$ and the polychromate polynomial, 
are reconstructible.
\end{thm}
\ifmargin
\else
\subsection{Reconstructing Neighborhoods}

Let $\mathbf{Nbd}_{G}(v,d)=\left(G_{d,v},l_{v}\right)$ be the isomorphism-type
of the pointed graph \[
Nbd_{G}(v,d)=\left(\left\langle B_{G}(v,d),E(G)\cap\left(B_{G}(v,d)\times B_{G}(v,d)\right)\right\rangle ,v\right).\]
 We call $v$ and $l_{v}$ the leading vertices of $Nbd_{G}(v,d)$
and $\mathbf{Nbd}_{G}(v,d)$ respectively. Let the neighborhood-deck
$\mathcal{N}(G)$ of a graph be the multiset \[
\mathcal{N}(G)=\left\{ \mathbf{Nb}\mathbf{d}_{G}(v,d)\mid v\in V(G)\right\} .\]
 We extend the definition of $\mathcal{N}$ to multisets of graphs.
Given a multiset of graphs $A$, \[
\mathcal{N}(A)=\bigcup_{G'\in A}\mathcal{N}(G').\]
Let $\wp(G)$ be the multiset of induced subgraphs of $G$ and let
\[
\mathcal{ND}(G)=\mathcal{N}(G_{1})\cup\cdots\cup\mathcal{N}(G_{n}).\]

Claim 4: $\mathcal{ND}(G)=\mathcal{N}\left(\left\{ G_{1},\ldots,G_{n}\right\} \right)$

Proof: Follows directly from the definitions. $\square$

Claim 2: Let $R\in\mathcal{ND}(G)$ where $R=\left(G',l\right)$.
Then there exists $v\in V(G)$ such that $R=\mathbf{Nbd}_{G}(v,d)$
or $R=\mathbf{Nbd}_{Nbd_{G}(v,d)-u}(v,d)$ for some $u\in V(Nbd_{G}(v,d))$. 

Proof: $R\in\mathcal{ND}(G)$ implies $R\in\mathcal{N}(G_{i})$ for
some $i$. There exists $v'\in V(G_{i})$ such that $R=\mathbf{Nbd}_{G_{i}}(v',d)$.
Let $v\in V(G)$ be the vertex in $G$ which corresponds to $v'$.
Then there are two cases:
\begin{enumerate}
\item If $v_{i}\in V(Nbd_{G}(v,d))$ then $R=\mathbf{Nbd}_{Nbd_{G}(v,d)-v_{i}}(v,d)$.
\item Otherwise, $\mathbf{Nbd}_{G}(v,d)=\mathbf{Nbd}_{G_{i}}(v',d)$. $\square$ 
\end{enumerate}
Claim 1: If $R\in\mathcal{N}(G)$ then $R\in\mathcal{ND}(G)$. 

Proof: Let $R=\mathbf{Nbd}_{G}(v,d)=(G',v)$. Notice $G'$ is a proper
induced subgraph of $G$ and thus is contained as an induced subgraph
in $n-|V(G')|>0$ cards. This implies $R\in\mathcal{N}(G_{i})$. $\square$

Let \[
\mathcal{B}_{I}(G)=\left\{ \stackrel{n-|V(\mathbf{Nbd}_{G}(v,d))|\mbox{ times}}{\overbrace{\mathbf{Nbd}_{G}(v,d),\ldots,\mathbf{Nbd}_{G}(v,d)}}\mid v\in I\right\} \]
 and let \[
\mathcal{C}_{I}(G)=\left\{ \mathbf{Nbd}_{Nbd_{G}(v,d)-u}(v,d)\mid u\not=v,v\in I,u\in V(Nbd_{G}(v,d))\right\} .\]

Conclusion 6: $\mathcal{ND}(G)=\mathcal{B}_{V(G)}(G)\cup\mathcal{C}_{V(G)}(G)$. 

\underbar{$Create-Neighborhoods(\mathcal{D}(G))$}:
\begin{enumerate}
\item Init $X$ to the empty multiset
\item Init $Y$ to the multiset $\mathcal{ND}(G)$
\item \label{alg:LoopFor}For $m=n-1$ to $m=1$ do,

\begin{enumerate}
\item \label{alg:LoopWhile}While there is a tuple $\left(N'_{1},\ldots,N'_{n-m}\right)\in Y^{n-m}$
of distinct isomorphic graphs of size $m$,

\begin{enumerate}
\item \label{alg:insert}Insert $N'_{1}$ to $X$
\item \label{alg:delete}Delete $N'_{1},\ldots,N'_{n-m}$ from $Y$
\item \label{alg:innerFor}For every proper $N'\subsetneq N'_{1}=(G',v)$
of size $m-1$ which contains $v$,

\begin{enumerate}
\item \label{alg:innerDelete}Delete $\mathbf{Nbd}_{N'}(G',v)$ from $Y$. 
\end{enumerate}
\end{enumerate}
\end{enumerate}
\end{enumerate}
Claim: 
\begin{enumerate}
\item In \ref{alg:LoopWhile}, if there is a graph $N'_{1}$ of size $m$,
then there are at least $n-m$ isomorphic copies of it.
\item \label{hyp2}After \ref{alg:insert} $X$ is a sub-multiset of the
neighborhood-deck of $G$.

Denote by $I_{X}$ the subset of $V(G)$ of vertices which correspond
to the neighborhoods in $X$. 

\item \label{hyp3} After \ref{alg:innerFor} $Y=\mathcal{B}_{V(G)-I_{X}}(G)\cup\mathcal{C}_{V(G)-I_{X}}(G)$
\end{enumerate}
Proof: By induction on the while loop:

Base: A maximal element of $Y$ is a neighborhood of $G$, because
$Y=\mathcal{B}(G)\cup\mathcal{C}(G)$. Thus, it appears at least $n-m$
times in $Y$. It $X$ contains after \ref{alg:insert} exactly one
element of $\mathcal{N}(G)$. 

Step: 
\begin{enumerate}
\item By the induction hypothesis \ref{hyp3}, the maximal element of $Y$
corresponds to a neighborhood of $G$ and occurs $n-m$ times. 
\item Before \ref{alg:insert}, $X$ was a sub-multiset of the neighborhood-deck
of $G$. After adding another neighborhood from $\mathcal{B}_{V(G)-I_{X}}(G)$
to $X$, \ref{hyp2} holds. 
\item We remove in \ref{alg:delete} the $m-n$ copies of the current neighborhood.
In \ref{alg:innerFor} we remove the $m$ graphs obtained from the
current neighborhood by means of deleting a single vertex. Thus, we
remove from $Y$ for some vertex $v$ all the subgraphs which result
in $\mathcal{ND}(G)$ from the neighborhood $\mathbf{Nbd}_{G}(v,d)$.
Thus, $Y=\mathcal{B}_{I-\left\{ v\right\} }(G)\cup\mathcal{C}_{I-\left\{ v\right\} }(G)$. 
\end{enumerate}

\subsection{Enter Logic}
\label{subse:Recog}
In this subsection we show an appliaction of the locality of First Order Logic, 
$FOL$, to the recognizability of graphs. 
As far as we know, there has been no attempt to use logical methods for reconstruction problems. 
Let us define some basics of definability of graphs in First Order Logic, $FOL$.
Let $\tau_{graph} = \left\langle \mathbf{E} \right\rangle$ where $\mathbf{E}$ is a binary relation symbol. 
Graphs are $\tau_{graph}$-structures $G=\left\langle  [n],E \right\rangle$, 
where $[n]=\{0,\ldots,n-1\}$ is the vertex set and 
$E\subseteq [n]\times [n]$ is the edge relation. 
The sentences of $FOL$ are the closure by the boolean connectives $\lor,\land,\neg$
and the quantifiers $\forall x_i, \exists x_i$ of the atomic formulas
$E(x_i,x_j)$ and $x_i = x_j$. E.g., the following formula
\[
  \phi_{graph}=\forall x \forall y \left( \neg E(x,x) \land \left( E(x,y)\to E(y,x)\right) \right)
\]
says a a $\tau_{graph}$ structure is a simple undirected graph. 

For a class of graphs $\mc{C}$ we say the sentence $\phi$ defines $\mc{C}$ if 
a graph $G$ belongs to $\mc{C}$ iff $\phi$ is true for $G$. 
The quantifier rank of a $FOL$ sentence $\phi$ denoted $qr(\phi)$ is 
the maximum depth of the nesting of quantifiers in $\phi$. 
Let $G=(V(G),E(G))$ be a graph. For a vertex $v\in V(G)$, 
let  $Nbd_G(d,v)$ denote the induced subgraph of those vertices of distance at most $d$ from $v$. 

\begin{defn}
Let $G_1,G_2$ be two graphs such that there exists a bijection $f:V(G_1)\to V(G_2)$ such that for every $v \in V(G_1)$, $Nbd_{G_1}(v,d)$ and $Ndb_{G_2}(f(v),d)$ are isomorphic. Then we write $G_1 \leftrightarrows_d G_2$. 
\end{defn}

\begin{defn}
A class of graphs $\mc{C}$ is {\em Hanf local} if  there exists $d$ such that for every two graphs $G_1,G_2\in \mc{C}$,
 $G_1 \leftrightarrows_d G_2$ implies that  $G_1 \in \mc{C}$ iff $G_2 \in \mc{C}$. 
\end{defn}

We use the following theorem, which follows from \cite{ar:FSV95}:
\begin{thm}
\label{thm:FSV95}
 Let $\mc{C}$ be the class of graphs which satisfy a $FOL$ formula $\phi$ of quantifier rank $k$. Then $\mc{C}$ is Hanf local and for $d\geq \frac{3^k-1}{2}$.
\end{thm}

\begin{lem}
Let $G$ and $H$ be graphs of diameter at least $d$  which are reconstructions of each other. Then for every class of graphs $\mc{C}$ which is $FOL$-definable by a sentence $\phi$ or quantifier rank $k\leq\log_3(n+1) -1$ it holds that $G\in\mc{C}$ iff $H\in \mc{C}$. 
 \end{lem}

\begin{lem}
Let $G$ be a graph of bounded degree $f$. Let $H$ be a reconstruction
of $G$. Then $G$ and $H$ agree on all $FOL$ sentences of quantifier
rank less than $\log_{3}\log_{f}\left(n-1\right)$. \end{lem}
\begin{proof}
If $G$ has maximal degree $\leq f$ then so does $H$, because the
degree sequence is reconstructible. Look at $Nbd(d,v).$ Since the
degree is bounded, the size of $Nbd(d,v)$ is bounded by $f^{d}.$
So, if $d=\log_{f}(n-1)$, then $f^{d}<n$ and thus by Kelly's lemma,
$G$ and $H$ have the same number of $d-$type neighborhood. By \cite{ar:FSV95}, $H$ and $G$ agree on all $FOL$
sentences $\Phi$ such that $3^{qr(\Phi)}\leq d$. So, $H$ and $G$
agree on all $FOL$ sentences such that $qr(\Phi)\leq\log_{3}\log_{f}(n-1)$. 
\end{proof}

\begin{thm}
Let $\mathcal{G}_{n}=\left(G_{1},G_{2},G_{3}\ldots\right)$ be a sequence
of graphs such that $|V(G_{i})|=i$ and there is $f\in\mathbb{N}$
that bounds the degree of all $G_{i}$'s. Let $\mathcal{C}$ be a
recognizable class of graphs. If there is a formula $\phi\in FOL$
such that for a graph $G\in\mathcal{C}$ of size $n$, $G\models\phi$
iff $G$ is isomorphic to $G_i$, then for every $n>n_{\phi}$, $G_{n}$ is reconstructible.\end{thm}

\begin{proof}
Let $n_{\phi}=3^{f^{qr(\phi)}}+1$ where $qr(\phi)$ is the quantifier
rank of $\phi$. If $|V(G)|=|V(G_{i})|>n_{\phi}$ then $G$ and $G_{i}$
agree on $\phi$, so $G\cong G_{i}$. 
\end{proof}

\begin{cor}
The following are reconstructible for large enough $n$: Clique $K_n$, Ladder $L_n$, book graph, gear graph, prism graph, anti-prism graph, Mo\"bius graph, crossed-prism graph and grid $G_{n_1\times n_2}$.
\end{cor}
\begin{proof}
We give the classes $\mc C$ and the defining $FOL$ formulae:
\begin{itemize}
\item Clique: $\mc C=CON$, $\phi=\forall x,y\,(x\not=y\leftrightarrow E(x,y))$
\item Ladder:  $\mathcal{C}=CON$, 
$\phi=\exists x,y,z,w,\,\left|\left\{ x,y,z,w\right\} \right|=4\land d(x)=d(y)=d(z)=d(w)=2\land E(x,y)\land E(z,w)$
\[
\land\forall t\notin\left\{ x,y,z,w\right\} ,d(t)=3\land\forall t,s,r,\,\neg E(t,s)\lor\neg E(s,r)\lor\neg E(r,t)\]
\[
\land\forall t,s,r,p,\,\left(E(t,s)\land E(s,r)\land E(r,p)\right)\to E(t,p).\]
where $d(x)$ is the degree of $x$.
\item Book graph: $\mathcal{C}=CON\land n\geq8$, 
$\phi=\exists x,y,(x\not=y)\land\forall z\notin\left\{ x,y\right\} ,d(z)=2\land(E(x,z)\lor E(y,z))\land(E(x,z)\leftrightarrow\neg E(y,z))\land$
$
\forall w\notin\left\{ x,y,z\right\} ,E(z,w)\to\left(E(x,z)\leftrightarrow\neg E(x,w)\right)$

\item Gear graph: $\mathcal{C}=CON$ and there is $G_{i}\in\mathcal{D}(G)$
such that $G_{i}\cong C_{n-1}$ and $n\geq8$, 
$\phi=\exists x,\, d(x)>3\forall y,z\not=x,\, d(y)\in\left\{ 2,3\right\} \land d(y)=2\leftrightarrow d(z)=3\land(d(y)=3\leftrightarrow E(y,x))$
\item Grid $G_{n\times n}$ for $n=p\cdot q$ and $p,q$ prime numbers -
$\mathcal{C}=CON$ and $\phi$ is defined as requiring that there
are 4 distinct vertices of degree $2$ and all the rest of degree
3 and 4, and by requiring that the every vertex belong to subgraph
of size $4,6$ or $9$ (for degree 2, 3 or 4 respectively) such that
the vertices in this subgraph have prescribed degrees (for example,
for $x$ of degree 2 we require a square $4-3-2-3$ where the first
4 is adjacent also to the last 3).
\item Grid $G_{n_1\times n_2}$ for every  (large enough) $n_1$ and $n_2$,  by taking the same defining formula as before and $\mathcal{C}=CON\land ALMOST-GRID(n_1,n_2)$, where $ALMOST-GRID(n_1,n_2)$ is the class of graphs that have in their deck a graph isomorphic to a grid that has had one of its corners deleted. This is obviously recognizable, and the grid is thus determined to be of size $n_1\times n_2$, which means it is $G_{n_1\times n_2}$. 
\item Prism graph, Mo\"bius graph, crossed prism graph and antiprism graph - the proof is similar to that of the ladder graph. 
\end{itemize}
\end{proof}

\fi 

\nocite{ar:BH77}
\nocite{pr:Bondy91}
\nocite{pr:Tutte79}

\begin{paragraph}{Acknowledgement}
\footnotesize
 I am grateful to my PhD supervisor J.A. Makowsky, for his advise and support,
and to I. Averbouch, L. Lyaudet and N. Zewi for useful discussions on issues related
to this work.
\end{paragraph}


\begin{thebibliography}{10}

\bibitem{ar:ABS04}
R.~Arratia, B.~Bollob{\'a}s, and G.~B. Sorkin.
\newblock The interlace polynomial of a graph.
\newblock {\em J. Comb. Theory, Ser. B}, 92(2):199--233, 2004.

\bibitem{pr:agm08}
I.~Averbouch, B.~Godlin, and J.~A. Makowsky.
\newblock A most general edge elimination polynomial.
\newblock In {\em WG}, volume 5344 of {\em LNCS}, pages 31--42, 2008.

\bibitem{pr:Bondy91}
J.~A. Bondy.
\newblock A graph reconstructor's manual.
\newblock In D.~Keedwell, editor, {\em Surveys in Combinatorics}, pages
  221--252. Cambridge Univ.\ Press, 1991.

\bibitem{ar:BH77}
J.~A. Bondy and R.~L. Hemminger.
\newblock Graph reconstruction -- a survey.
\newblock {\em Journal of Graph Theory}, 1(3):227--268, 1977.

\bibitem{ar:FSV95}
Fagin, Stockmeyer, and Vardi.
\newblock On monadic {NP} vs monadic co-{NP}.
\newblock {\em INFCTRL: Information and Computation (formerly Information and
  Control)}, 120, 1995.

\bibitem{ar:DPT03}
A.~Poenitz K.~Dohmen and P.~Tittmann.
\newblock A new two-variable generalization of the chromatic polynomial.
\newblock {\em Discrete Mathematics \& Theoretical Computer Science},
  6(1):69--90, 2003.

\bibitem{ar:Kocay81}
W.~L. Kocay.
\newblock An extension of kelly's lemma to spanning subgraphs.
\newblock {\em Congressus Numeratium}, 31:109--120, 1981.

\bibitem{ar:MW78}
B.~Manvel and J.~M. Weinstein.
\newblock Nearly acyclic graphs are reconstructible.
\newblock {\em Journal of Graph Theory}, 2(1):25--39, 1978.

\bibitem{ar:nb99}
S.~D. Noble and D.~J.~A. Welsh.
\newblock A weighted graph polynomial from chromatic invariants of knots.
\newblock {\em Ann. Inst. Fourier}, 49(3):1057--1087, 1999.

\bibitem{pr:Tutte79}
W.~T. Tutte.
\newblock All the king's horses.
\newblock In J.~A. Bondy and U.~S.~R. Murty, editors, {\em Graph Theory and
  Related Topics}, pages 15--34, 1979.

\bibitem{ar:Weinstein75}
J.~M. Weinstein.
\newblock Reconstructing colored graphs.
\newblock {\em Pacific Journal of Mathematics}, 57(1):307--314, 1975.

\end{thebibliography}

\end{document}